\documentclass[a4paper]{article}
\usepackage[english]{babel}
\usepackage[cp1251]{inputenc}
\usepackage[colorlinks,linkcolor=blue,citecolor=blue]{hyperref}
\usepackage{latexsym,amsfonts,amssymb,amsthm,amsmath,amsbsy,graphicx}
\newtheorem{theorem}{Theorem}
\newtheorem{lema}{Lemma}
\newtheorem{slid}{Corollary}
\begin{document}
\thispagestyle{myheadings} \markboth{ On the exit from a finite
interval}{Theory of Stochastic Processes, Vol. 11 (27), no. 3-4,
2005, pp.71-81}
\bigskip
 {\noindent \Large\bf\sc    On the exit from a finite interval for the risk processes with
stochastic premiums}\footnotemark[1] \footnotetext{This is an
electronic reprint of the original article published in Theory of
Stochastic Processes, Vol. 11 (27), no. 3-4, 2005. This reprint
differs from the original in pagination and typographic detail.}

\bigskip
\bigskip
 { \bf D.V. Gusak,\footnote{Institute of Mathematics, Ukrainian National
 Academy of Science, 3 Tereshenkivska str.,
252601 Kyiv, Ukraine. \phantom{\quad \quad }
 \href{mailto:random@imath.kiev.ua}{random@imath.kiev.ua}
 } E.V.~Karnaukh
 \footnote{Department of Probability and Mathematical Statistics,
 Kyiv National
University, 64~Vladimirs\-kaya str., 252017 Kyiv, Ukraine.
\phantom{\quad \quad }
 \href{mailto:kveugene@mail.ru}{kveugene@mail.ru}
 }}\hskip 6 cm UDC 519.21
\begin{center}
\bigskip
\begin{quotation}
\noindent  {\small In this article the almost semi-continuous
step-process $\xi (t)$ is considered. The conditional characteristic
functions of the jumps of $\xi (t)$ have the form $\mathrm{E}\left [
e^{i\alpha \xi _k}/\xi_k>0\right ]=c(c-i\alpha )^{-1}$. For such
processes the boundary functionals connected with the exit from the
finite interval are investigated.}
\end{quotation}
\end{center}
\bigskip

The problems on the exit from the finite interval for the process
$\xi (t)$ $(t\geq 0, \xi (0)=0)$ with stationary independent
increments were considered by many authors (see, for
example~\cite[ch.~IV, \S~2]{1}). In~\cite{1} the joint distributions
of extrema and the distributions of the values of the process up to
the exit from the interval were expressed in terms of rather
complicate series of the "convolutions" of
$$\Gamma ^{\pm }(s,x,y)=\mathrm{E}\,\left [  e^{-s\tau ^{\pm }(\pm x)},
\gamma ^{\pm }(\pm x)\leq y\right ],$$ where
\begin{gather*} \tau
^{\pm }(\pm x)=\inf\left\{t>0:\pm \xi (t)> x\right\},
\gamma ^{\pm }(\pm x)=\pm \xi (\tau ^{\pm }(\pm x))\mp x,\;x>0.
\end{gather*}
\par Simpler relations for the Wiener processes are established
in~\cite[p. 463]{1} and in~\cite[\S~27]{2}. In~\cite{3}~-~\cite{6},
the mentioned problems are investigated for the semi-continuous
processes $\xi (t)$ ($\xi (t)$ have jumps of one sign). For these
processes, in~\cite{7}~-~\cite{8} the density of distribution of
$\xi (t)$ up to the exit from the interval was represented in terms
of the resolvent functions $R_s(x)$ (introduced by V.S.~Korolyuk
in~\cite{3}).
\par We'll consider the compound Poisson process
$$\xi (t)=\sum_{k\leq \nu (t)}\xi _k,$$
where $\nu (t)$ is the Poisson process with the rate $\lambda>0$.
The distributions of $\xi _k$ satisfy the next condition ($F(x)$ is
the cumulative distribution function)
\begin{equation}
  \mathrm{P}\left\{\xi _k<x\right\}=qF(x)I\left\{x\leq
  0\right\}+(1-pe^{-cx})I\left\{x>0\right\},\;c>0,\,p+q=1.
\end{equation}
The process $\xi (t)$ is the almost upper semi-continuous piecewise
constant process. We can represent $\xi (t)$ as the claim surplus
process $\xi (t)=C(t)-S(t)$ with the stochastic premium function
$$
C(t)=\sum_{k\leq \nu _1(t)}\eta  _k,\;\eta  _k>0,
\;\mathrm{E}\,e^{i\alpha \eta  _k}= \frac{c}{c-i\alpha }, \;c>0,
$$
and with the process of claims $S(t)=\sum_{k\leq \nu_2 (t)}\xi'
_k,\; \xi' _k> 0 $. $\nu_1 (t),\;\nu _2(t)$ - are independent
Poisson processes with the rates $\lambda_1,\,\lambda _2 >0$,
$\lambda _1+\lambda _2=\lambda $ (for details see~\cite{8} ).
\par Note, that  $C(t)\rightarrow 0$ and
$\xi (t)\rightarrow -S(t)$ as $c\rightarrow \infty$, where $-S(t)$
is the non-increasing process.
\par Let $C_c(t)$ be the process with the cumulant
$$
\psi _c(\alpha )=\lambda _c\left(\frac{c}{c-i\alpha }-1 \right),
\quad \lambda _c=ac,\;a>0,
$$
then $\psi _c(\alpha )\underset{c\to\infty}{\longrightarrow}i\alpha
a$, consequently $C_c(t) \underset{c\to\infty}{\longrightarrow}at$,
and $ \xi _{c}(t)=C_c(t)-S(t)\rightarrow \xi ^{0}(t)=at-S(t)$, where
the limit process $\xi ^{0}(t)$ is the classical upper
semi-continuous risk process with the non-stochastic premium
function $C(t)=at$.
\par Let $\theta _s$ be the exponentially distributed random variable $(\hbox{P}\{\theta
_s>t\}=e^{-st};\;s,t>0)$. Then the randomly stopped process $\xi
(\theta _s)$ have the characteristic function (ch.f.)
$$
\varphi (s,\alpha )=\hbox{E}e^{i\alpha \xi (\theta _s)}=\frac{s}{s-\psi (\alpha )},
$$
where
\begin{equation}
  \psi (\alpha )=\lambda p(c(c-i\alpha )^{-1}-1)+\lambda q(\varphi (\alpha )-1),\; \varphi
(\alpha )=\int_{-\infty}^{0}e^{i\alpha x}dF(x).
\end{equation}
\par Let us denote the first exit time from the interval
$\left (  x-T,x\right )$, $0<x<T$, $T>0$:
$$
\tau (x,T)=\inf\left\{t>0: \xi (t)\notin (x-T,x)\right\},
$$
and the events
$$
A_+(x)=\left\{\omega : \xi (\tau (x,T))\geq
x\right\},\;A_-(x)=\left\{\omega : \xi (\tau (x,T))\leq x-T\right\}.
$$
Then
$$
  \tau (x,T)\dot{=}\begin{cases}
  \tau ^+(x,T)=\tau ^+(x),\;\omega \in A_+(x);\\
  \tau ^-(x,T)=\tau ^-(x-T),\;\omega \in A_-(x).
  \end{cases}
$$
 Overshoots at the moments of the exit from the interval we denote by
the following relations:
\begin{gather*} \gamma ^-_T(x)=x-T-\xi
(\tau^-(x,T)),\;
 \gamma ^+_T(x)=\xi (\tau ^+(x,T))-x.
\end{gather*}
 \par The main task of our paper is the finding the next moment generating functions (m.g.f.) of
the functionals connected with the exit from the interval.
\begin{gather*}
Q(T,s,x)=\mathrm{E}\,e^{-s \tau (x,T)} ,\\
Q^T(s,x)=\mathrm{E}\,\left [  e^{-s\tau ^+(x,T)},\,A_+(x)\right ],\\
Q_T(s,x)=\mathrm{E}\,\left [  e^{-s\tau ^-(x,T)},\,A_-(x)\right ],\\
V^{\pm}(s,\alpha ,x,T)=\mathrm{E}\,\left [  e^{i\alpha \gamma
^{\pm}_T(x)-s\tau
^{\pm}(x,T)},\,A_{\pm}(x)\right ],\\
V_{\pm}(s,\alpha ,x,T)=\mathrm{E}\,\left [  e^{i\alpha \xi (\tau
^{\pm }(x,T))-s\tau
^{\pm}(x,T)},\,A_{\pm}(x)\right ],\\
V(s,\alpha ,x,T)=\mathrm{E}\,\left [ e^{i\alpha \xi
(\theta_s)},\,\tau (x,T)>\theta_s\right ],
\end{gather*}
\par Let us denote the extrema $\displaystyle\xi ^{\pm}(t)=\sup_{0\leq s \leq t}(\inf)\xi (s)$,
 $\displaystyle\xi ^{\pm}=\sup_{0\leq s < \infty}(\inf)\xi(s),$
the joint distribution of $\left\{\xi (\theta_s),\xi^+
(\theta_s),\xi^- (\theta_s)\right\}$:
$$
  \begin{aligned}
  H_s(T,x,y)&=\mathrm{P}\left\{\xi (\theta_s)<y,\xi ^+(\theta _s)<x,\xi ^-(\theta_s)>x-T\right\}\\
  &=\mathrm{P}\left\{\xi (\theta_s)<y,\tau (x,T)>\theta_s\right\},
  \end{aligned}
$$
and
$$
P_{\pm }(s,x)=\mathrm{P}\left\{\xi ^{\pm
}(\theta_s)<x\right\},\;x\gtrless 0, \;p_{\pm
}(s)=\mathrm{P}\left\{\xi ^{\pm }(\theta_s)=0\right\},\;q_{\pm
}(s)=1-p_{\pm }(s);
$$
$$
\varphi _{\pm }(s,\alpha )=\pm \int^{\pm \infty}_{0}e^{i\alpha x}dP_{\pm }(s,x),
$$
$$
T^{\pm }(s,x)=\mathrm{E}\,\left [  e^{-s\tau ^{\pm }(x)},\tau ^{\pm
}(x)<\infty\right ],\;x\gtrless 0.
$$
\begin{lema} For the process $\xi(t)$  with cumulant~{\rm{(2)}} the main
factorization identity is represented by relations
\begin{equation}
  \varphi (s,\alpha )=\varphi_+ (s,\alpha )\varphi_- (s,\alpha ), \ \Im\alpha=0;
\end{equation}
\begin{equation}
   \varphi_+(s,\alpha
)=\dfrac{p_+(s)(c-i\alpha)}{\rho_+(s)-i\alpha},
\end{equation}
where $\rho_+(s)=cp_+(s)$ is the positive root of Lundberg's equation $\psi (-i r)=s$, $s>0$.
\begin{equation}
 \mathrm{P}\left\{\xi ^+(\theta_s)>x\right\}= T^+(s,x)=q_+(s)e^{-c\rho_+(s)x}, x>0.
\end{equation}
If $m>0:$
\begin{equation}
  \lim_{s\rightarrow 0}\rho _+(s)s^{-1}=\rho' _+(0)=m^{-1},\;
  \lim_{s\rightarrow 0}P_-(s,x)=\mathrm{P}\left\{\xi
  ^-<x\right\},\,x<0.
\end{equation}
If $m<0:$
\begin{equation}
\lim_{s\rightarrow 0}\rho _+(s)=\rho _+ >0;\;\lim_{s\rightarrow
0}s^{-1}\mathrm{P}\left\{\xi ^-(\theta_s)>x
\right\}=\mathrm{E}\,\tau ^-(x),\;x<0.
\end{equation}
\noindent If $\sigma _1^{2}=D\xi (1)<\infty$ and $m=\lambda
\left(pc^{-1}-q\widetilde{F}(0)\right)=0$
$\left(\widetilde{F}(0)=\int_{-\infty}^{0}F(x)dx \right)$, then
\begin{gather} \lim_{s\rightarrow 0}\rho
_+(s)s^{-1/2}=\frac{\sqrt{2}}{\sigma _1};\;\lim_{s\rightarrow
0}s^{-1/2}P'_-(s,x)=f_0(x),\;x<0,\nonumber
 \\
f_0(x)=k_0\frac{\partial}{\partial
x}\left(\int_{0}^{\infty}\mathrm{P}\left\{\widetilde{\xi
}_0(t)<x\right\}
  dt\right)=-k_0\frac{\partial}{\partial x}\mathrm{E}\,\tau _0(x),x<0;
\end{gather}
where $k_0=c\sigma _1\left(\sqrt{2}\right)^{-1}$, $\tau _0(x)=\inf\left\{t>0:\,\widetilde{\xi
}_0(t)<x\right\}$, $x<0$; $\widetilde{\xi }_0(t)$ is the decreasing process with the spectral
measure
$$\Pi
_0(dx)=\lambda q\left(cF(x)dx +dF(x)\right),\;x<0.$$
\end{lema}
\begin{proof} Relations~\rm{(3)}~-~\rm{(7)} were proved
in~\cite{7}~-~\cite{8}. If $m=0$ $\left(p=cq\widetilde{F}(0)
\right)$, then
$$
\varphi (s,\alpha )=\frac{s(c-i\alpha )}{s(c-i\alpha )
-i\alpha \lambda (p-q\widetilde{F}(\alpha )(c-i\alpha ))},\;
\widetilde{F}(\alpha )=\int_{-\infty}^{0}e^{i\alpha x}F(x)dx.
$$
On the basis of factorization identity~\rm{(3)} as $s\rightarrow 0$,
we get
\begin{gather*} \frac{1}{\sqrt{s}}\varphi _-(s,\alpha
)=\frac{\sqrt{s}}{p_+(s)}\frac{\rho _+(s)-i\alpha }{s(c-i\alpha
)-i\alpha \lambda \left(p-q\widetilde{F}(\alpha )(c-i\alpha )
\right)}\rightarrow \widetilde{f}_0(\alpha
),\\
\widetilde{f}_0(\alpha )=\frac{c \sigma _1}{\sqrt{2}}\frac{1}{-\lambda q\left [
\left(\widetilde{F}(\alpha )-\widetilde{F}(0) \right)c+\varphi (\alpha ) -1\right ]}=\frac{c \sigma
_1}{\sqrt{2}}\frac{1}{-\widetilde{\psi }_0(\alpha )},\\
\widetilde{\psi }_0(\alpha )=\int_{-\infty}^{0}\left(e^{i\alpha x}-1 \right)\Pi _0(dx), \;\Pi
_0(dx)=\lambda q\left(cF(x)dx+dF(x) \right),\;x<0.
\end{gather*}
Let's denote
$$
\varphi _0(s,\alpha )=\mathrm{E}\,e^{i\alpha \widetilde{\xi}
_0(\theta_s)}= \frac{s}{s-\widetilde{\psi} _0(\alpha )},
$$
where $\widetilde{\xi} _0(t)$ is the decreasing process with the cumulant $\widetilde{\psi}
_0(\alpha )$. Since
$$
\frac{c \sigma _1}{\sqrt{2}}\varphi _0(s,\alpha )s^{-1}\rightarrow \widetilde{f}_0(\alpha )
=\int_{-\infty}^{0}e^{i\alpha x}f_0(x)dx,\; s\rightarrow 0,
$$
we get that
$$
f_0(x)=k_0\frac{\partial}{\partial x}\left(\int_{0}^{\infty}
\mathrm{P}\left\{\widetilde{\xi} _0(t)<x\right\}dt\right),
$$
or
$$
-f_0(x)=k_0\frac{\partial}{\partial
x}\int_{0}^{\infty}\mathrm{P}\left\{\tau _0(x)>t\right\}dt
=k_0\frac{\partial}{\partial x}\mathrm{E}\,\tau _0(x),\;x<0.
$$
\end{proof}
Let's introduce the set of boundary functions on the interval $I\subset (-\infty,\infty)$:
$$\mathfrak{L}(I)=\left\{G(x):\int_{I}|G(x)|dx<\infty\right\},$$
 and the set of integral transforms:
$$\mathfrak{R}^0(I)=\left\{g^0(\alpha ):g^0(\alpha )=C+\int_{I}e^{i\alpha x}G(x)dx\right\}.$$
Let's denote the projection operations on
$\mathfrak{R}^0((-\infty,\infty))$ by the next relations
$$\gathered
\left [  g^0(\alpha )\right ]_{I}=\int_{I}e^{i\alpha x}G(x)dx,\;\;
\left [  g^0(\alpha )\right ]^0_{I}=C+\int_{I}e^{i\alpha x}G(x)dx,\\
\left [  g^0(\alpha )\right ]_{-}=\left [  g^0(\alpha )\right
]_{(-\infty,0)},\;\left [  g^0(\alpha )\right ]_{+}=\left [
g^0(\alpha )\right ]_{(0,\infty)}.
\endgathered$$
\par The main results of our paper are included in the next two assertions.
\begin{theorem} For the process $\xi (t)$ with cumulant~{\rm{(2)}} $Q^T(s,x)$
has the next form $(0<x<T)$
\begin{multline}
  Q^T(s,x)=q_+(s)e^{-\rho _+(s)x}\int_{x-T}^{0}e^{\rho _+(s)y}dP_-(s,y)\times\\
  \times\left [  e^{-\rho _+(s)T}\int_{-\infty}^{-T}e^{c\left(T+y \right)}dP_-(s,y)
  +\int_{-T}^{0}e^{\rho _+(s)y}dP_-(s,y)\right
  ]^{-1}.
\end{multline}
\end{theorem}
\begin{theorem} For the process $\xi (t)$ with cumulant~{\rm{(2)}}
the joint distributions of
\par \noindent $\left \{\tau ^+(x,T),\gamma ^+_T(x)\right\}$ and $\left \{\tau ^+(x,T),\xi(\tau
^+(x,T))\right\}$ are determined by the next relations
\begin{equation}
  \begin{cases}
    V^+(s,\alpha ,x,T)=\displaystyle\frac{c}{c-i\alpha }Q^T(s,x),\;0<x<T, \\
     \displaystyle V_+(s,\alpha ,x,T)=e^{i\alpha x}V^+(s,\alpha ,x,T)
      =\frac{c\; e^{i\alpha x}}{c-i\alpha
      }Q^T(s,x).
  \end{cases}
\end{equation}
The ch.f. of $\xi (\theta_s)$ before the exit time from the interval
has the form
\begin{equation}
\begin{split}
  V(s,\alpha ,x,T)&=\varphi _+(s,\alpha )\left [  \varphi _-(s,\alpha )
  \left(1-V_+(s,\alpha, x,T) \right)\right ]_{ [  x-T,\infty)}\\
  &=\varphi _+(s,\alpha )\left [  \varphi _-(s,\alpha )
  \left(1-c\; e^{i\alpha x}(c-i\alpha )^{-1}Q^T(s,x) \right)\right ]_{ [  x-T,\infty)},
\end{split}
\end{equation}
the corresponding distribution has the next density
$(x-T<z<x,\;z\neq 0)$
\begin{multline}
  h_s(T,x,z)=\frac{\partial}{\partial z}H_s(T,x,z)=\\
 \shoveleft{= \left( p_+(s)P'_-(s,z)
 -q_+(s)\rho _+(s)\int_{z}^{0}e^{\rho _+(s)(y-z)}dP_-(s,y)\right)I\left\{z<0\right\}+}\\
+\rho _+(s)Q^T(s,x)\int_{z-x}^{0}e^{\rho _+(s)(y-(z-x))}dP_-(s,y),
\end{multline}
and the next atomic probability
$$\mathrm{P}\left\{\xi (\theta_s)=0,\tau (x,T)>\theta_s\right\}=
\mathrm{P}\left\{\xi
(\theta_s)=0\right\}=p_-(s)p_+(s)=\frac{s}{s+\lambda }.$$
\end{theorem}
\begin{proof}
First, let us prove Theorem 1. From the stochastic relations for
$\tau ^+(x,T)$, $\gamma ^+_T(x)$ ($\xi =\xi _1$ have the cumulative
distribution function $F_1(x)$, $\zeta $ - the moment of the first
jump of $\xi (t)$):
\begin{gather*}
  \tau ^+(x,T)\dot{=}
  \begin{cases}
    \zeta ,\; \xi >x, \\
    \zeta +\tau ^+(x-\xi ,T),\;x-T<\xi <x,
  \end{cases}\\
\gamma  ^+_T(x)\dot{=}
  \begin{cases}
    \xi -x ,\; \xi >x, \\
    \gamma ^+_T(x-\xi ),\;x-T<\xi <x,
  \end{cases}
\end{gather*}
we have the next equation for $V^+(s,\alpha ,x)=V^+(s,\alpha ,x,T)$
\begin{equation}
  (s+\lambda )V^+(s,\alpha ,x)=\frac{\lambda p c}{c-i\alpha }e^{-c x}+\lambda \int_{x-T}^x
V^+(s,\alpha ,x-z)dF_1(z),\;0<x<T.
\end{equation}
If $\alpha =0$, then from~\rm{(13)} we obtain the equation for $Q^T(s,x)$
\begin{equation}
  (s+\lambda )Q^T(s,x)=\lambda p e^{-c x}+\lambda \int_{x-T}^x Q^T(s,x-z)dF_1(z),\;0<x<T.
\end{equation}
Since $\mathrm{P}\left(A_+(x) \right)=1$ for $x<0$, then we have the
next boundary conditions
$$
Q^T(s,x)=
  \begin{cases}
    0,\; x>T, \\
    1,\; x<0.
  \end{cases}
$$
After the replacement
$$\overline{Q}\phantom{|}^T(s,x)=1-Q^T(s,x),$$
relation~\rm{(14)} yields the equation for
$\overline{Q}\phantom{|}^T(s,x)$ $\left(0<x<T \right)$
$$
  (s+\lambda )\overline{Q}\phantom{|}^T(s,x)=s+\lambda F(x-T)+\lambda
  \int_{0}^{T}\overline{Q}\phantom{|}^T(s,z)F'_1(x-z)dz,
$$
which after prolonging for $x>0$ has the form:
\begin{equation}
  (s+\lambda )\overline{Q}\phantom{|}^T(s,x)=
  sC(x)+\lambda \int_{-\infty}^{\infty}\overline{Q}\phantom{|}^T(s,z)F'_1(x-z)dz
  +C_T^>(s,x),
\end{equation}
$$
C(x)=I\left\{x>0\right\},\;C_T^>(s,x)=\overline{C}_T(s)e^{-c x},\;x>0,$$
\begin{equation}
\overline{C}_T(s)=\lambda p\left [
e^{cT}-c\overline{Q}\phantom{|}^{*}_s(T)\right
],\;\overline{Q}\phantom{|}^{*}_s(T)=\int_{0}^{T}e^{c
x}\overline{Q}\phantom{|}^T(s,x)dx.
\end{equation}
Let's introduce the function $C_\epsilon (x)=e^{-\epsilon x}C(x)$,
$x>0$, and consider instead of~\rm{(15)} the equation for
$Y_\epsilon (T,s,x)$ ($\epsilon >0$):
\begin{equation}
  (s+\lambda )Y_\epsilon (T,s,x)=sC_\epsilon (x)+\lambda \int_{-\infty}^{\infty}Y_\epsilon
  (T,s,x-z)dF_1(z)+C^>_T(s,x),\;x>0.
\end{equation}
\par Denote
\begin{gather*}
  y_\epsilon (T,s,\alpha )=\int_{0}^{\infty}e^{i\alpha x}Y_\epsilon (T,s,x)dx,\;
  \widetilde{C}_{\epsilon }(\alpha )=\int_{0}^{\infty}e^{i\alpha x}C_{\epsilon }(x)dx,
   \\
  \widetilde{C}_T(s,\alpha)=\int_{0}^{\infty}e^{i\alpha x}C^>_T(s,x)dx.
\end{gather*}
After integral transform from~\rm{(17)} we obtain the next equation
$$
(s-\psi (\alpha ))y_{\epsilon }(T,s,\alpha )=s\widetilde{C}_\epsilon (\alpha
)+\widetilde{C}_T(s,\alpha )-\left [ y_\epsilon (\alpha )\varphi
(\alpha )\right ]_-
$$
or
\begin{equation}
s y_\epsilon (T,s,\alpha )\varphi ^{-1}(s,\alpha )=s\widetilde{C}_\epsilon (\alpha
)+\widetilde{C}_T(s,\alpha )-\left [  y_\epsilon (\alpha )\varphi (\alpha )\right ]_-.
\end{equation}
After using the factorization decomposition~\rm{(3)} and the
projection operation $\left [ \,\right]_{+}$, relation~\rm{(18)}
yields
$$
s y_\epsilon (T,s,\alpha )\varphi ^{-1}_+(s,\alpha )= \left [\varphi
_-(s,\alpha )\left(s\widetilde{C}_\epsilon (\alpha
)+\widetilde{C}_T(s,\alpha ) \right)  \right ]_+
$$
or
\begin{equation}
  sy_\epsilon (T,s,\alpha )=\varphi _+(s,\alpha )\left [\varphi _-(s,\alpha )
  \left(s\widetilde{C}_\epsilon (\alpha )+\widetilde{C}_T(s,\alpha ) \right)  \right
  ]_{+}.
\end{equation}
By inverting of~\rm{(19)}, we obtain
\begin{equation}
  sY_\epsilon (T,s,x)=s\int_{0}^{x}B_\epsilon (x-y)dP_+(s,y)
  +\int_{0}^{x}B(s,x-y,T)dP_+(s,y),
\end{equation}
\begin{gather*}
 \begin{aligned} B_\epsilon
(x)&=\int_{-\infty}^{x}e^{-\epsilon
(x-y)}dP_-(s,y)=\int_{-\infty}^{0}e^{-\epsilon
(x-y)}dP_-(s,y)=e^{-\epsilon x}\mathrm{E}\,e^{\epsilon \xi^-
(\theta_s)},
\end{aligned}\\
 B(s,x,T)=\overline{C}_T(s)\int_{-\infty}^{x-T}e^{-c(x-y)}dP_-(s,y),\;x>0.
\end{gather*}
 Taking into account that $C_\epsilon (x)\rightarrow I\left\{x>0\right\}$ as $\epsilon
\rightarrow 0$, then $Y_\epsilon (T,s,x)\rightarrow
\overline{Q}\phantom{|}^T(s,x)$ as $\epsilon \rightarrow 0$,
$0<x<T$. So Eq.~\rm{(20)} yields
$$
  s\overline{Q}\phantom{|}^T(s,x)=sP_+(s,x)+p_+(s)B(s,x,T)
  +\int_{+0}^{x}B(s,x-z,T)P'_+(s,z)dz.
$$
Taking into account that
\begin{multline*} q_+(s)\rho
_+(s)\int_{0}^{x}\int_{-\infty}^{z-T}e^{-c(z-y)}dP_-(s,y)
e^{-\rho _+(s)(x-z)}dz=\\
\begin{aligned}
 =&q_+(s)\rho _+(s)\int_{-\infty}^{x-T}e^{-\rho
_+(s)x+cy}dP_-(s,y)\int_{\max(0,y+T)}^{x}e^{-cq_+(s)z}dz\\
=&p_+(s)\biggl [  \int_{-\infty}^{-T}e^{cy-\rho _+(s)x}dP_-(s,y)+
\end{aligned}
\\
+\int_{-T}^{x-T}e^{\rho _+(s)(y+T-x)-cT}dP_-(s,y) -\int_{-\infty}^{x-T}
e^{-c(x-y)}dP_-(s,y)\biggr
],
\end{multline*}
we have
\begin{multline*}
s\overline{Q}\phantom{|}^T(s,x)=sP_+(s,x)+p_+(s)\overline{C}_T(s)
e^{-\rho _+(s)x}\times\\
\times\biggl[ \int_{-\infty}^{-T}e^{cy}dP_-(s,y)+\int_{-T}^{x-T}
e^{-cT+\rho
_+(s)(y+T)}dP_-(s,y)\biggr].
\end{multline*}
From the last equation we can find $\overline{C}_T(s)$, and
$\overline{Q}\phantom{|}^*_s(T)$, and then get~\rm{(9)}.
\end{proof}
Let's note, that $Q^T(s,x)\rightarrow \overline{P}_+(s,x)$, as
$T\rightarrow \infty$ and $Q^T(s,x)\rightarrow 0$, as $c\rightarrow
\infty$. If we consider, instead of $\xi(t)$, the process $\xi
_c(t)=C_c(t)-S(t)$, then relation (9) yields
\begin{multline*}
Q^T_c(s,x)=q_+^c(s)\mathrm{E}\,\left [  e^{\rho ^c_+(s) (\xi
_c^-(\theta_s)+T-x)},
\xi _c^-(\theta_s)+T-x>0\right ]\times\\
\times\left(\mathrm{E}\,\left [  e^{c(\xi _c^-(\theta_s)+T)}, \xi
_c^-(\theta_s)+T<0\right ]
 + \mathrm{E}\,\left [  e^{\rho ^c_+(s)(\xi _c^-(\theta_s)+T)},
 \xi _c^-(\theta_s)+T>0\right ]
  \right)^{-1}.
\end{multline*}
Taking into account that for $x>0$: $\mathrm{P}\left\{\xi
^+_c(\theta_s)>x\right\}= q_+^c(s)e^{-\rho _+^c(s)x}
\underset{c\to\infty}{\longrightarrow}e^{-\rho^+ _0(s)x}$, where
$\rho^+ _0(s)$ is the positive solution of the equation
$$\psi ^0(-ir):=ar-\lambda_2
\left(\int_{-\infty}^0 e^{rx}dF(x)-1\right)=0,$$ we get
$Q^T_c(s,x)\rightarrow Q_\infty^T(s,x)$ as $c\rightarrow \infty$. If
we denote
$$\xi^0_{\pm}(t)=\sup_{0\leq u\leq t}(\inf)\xi^0(u),$$
then
\begin{equation*}
\begin{split} Q_\infty^T(s,x)&=\mathrm{E}\left [  e^{\rho^+ _0(s) (\xi
_-^0(\theta_s)+T-x)},
                                      \xi _-^0(\theta_s)
                                      +T-x>0\right ]
\left(\mathrm{E}\,\left [  e^{\rho^+ _0(s) (\xi _-^0(\theta_s)+T)},
                                     \xi _-^0(\theta_s)+T>0\right ]
  \right)^{-1}\\
&=\int^{T-x}_{0}e^{\rho^+ _0(s)(T-x-y)}
        d\mathrm{P}\left\{-\xi _-^0(\theta_s)<y\right\}\left(\int^{T}_{0}e^{\rho^+ _0(s)(T-y)}
        d\mathrm{P}\left\{-\xi _-^0(\theta_s)<y\right\} \right)^{-1}\\
&=R_s(T-x)R^{-1}_s(T),
\end{split}
\end{equation*}
where the last relation is the well-known formula(see~\cite{3}) for
the upper semi-continuous processes.

\begin{proof} Consider the proof of the second theorem. The first relation in~\rm{(10)} follows
from equations~\rm{(13)} and~\rm{(14)}. The second relation follows
from the first one. The first equality in~\rm{(11)} was proved
in~\cite{9}. After inverting~(11), we get
\begin{equation}
\begin{split}
h_s(T,x,z)=&p_+(s)\frac{\partial}{\partial
z}P_-(s,z)I\left\{z<0\right\}+
  q_+(s)\rho _+(s)\int_{x-T}^{\min\{z,0\}}e^{-\rho _+(s)(z-y)}dP_-(s,y)-\\
&-Q^T(s,x)\biggl[ p_+(s)\frac{\partial}{\partial z}
 \mathrm{P}\left\{\xi ^-(\theta_s)+\theta '_c+x\leq z\right\} +\\
& +q_+(s)\rho _+(s)\int_{x-T}^{z}e^{-\rho _+(s)(z-y)}
  dP\left\{\xi ^-(\theta_s)+\theta '_c+x<z\right\}\biggr].
\end{split}
\end{equation}
 Using the integral transform of~(21) with respect to the
distribution of $\theta '_c$ we get formula~\rm{(12)}.
\end{proof}
\begin{slid} For the joint distribution $\left\{\tau
^-(x,T),\xi (\tau ^-(x,T))\right\}$ we have
\begin{equation}
  s\mathrm{E}\,\left [  e^{-s\tau ^-(x,T)},\xi (\tau ^-(x,T))< z,\;A_-(x)\right ]=
    \int_{x-T}^{x}\Pi_-(z-y)dH_s(T,x,y),\; z\leq x-T,
\end{equation}
where $H_s(T,x,y)$ is determined by its density~{\rm (12)} and $ \Pi
_-(x)=\int_{-\infty}^{x}\Pi (dy),\;x<0.$

The probability of the lack of exit (non-exit) from the interval
$\left ( x-T,x\right )$ has the form
\begin{multline} \mathrm{P}\left\{\tau
(x,T)>\theta_s\right\}=\mathrm{P}\left\{\xi
^-(\theta_s)>x-T \right\}-\\
 -Q^T(s,x)\Biggl[  \int_{-\infty}^{-T}e^{c(z+T)}dP_-(s,z)+
 \mathrm{P}\left\{\xi^-(\theta_s)>-T \right\}\Biggr ].
\end{multline}
%
The m.g.f. for $\tau (x,T)$ and $\tau ^-(x,T)$ are determined in the following way
\begin{equation}
  \begin{cases}
    Q(T,s,x)=1-\mathrm{P}\left\{\tau (x,T)>\theta_s\right\},\; 0<x<T,\\
    Q_T(s,x)=Q(T,s,x)-Q^T(s,x),\; 0<x<T.
  \end{cases}
\end{equation}
\end{slid}
\begin{proof}Formula~\rm{(22)} follows from~\cite[Theorem 7.3]{6}. By
substitution~\rm{(12)} in~\rm{(22)}, we obtain the relation in terms
of $Q^T(s,x)$ and the truncated distribution of
$\xi^-(\theta_s)+\theta'_c$. Taking into account that
\begin{multline*}
\begin{aligned} \mathrm{P}\left\{\tau (x,T)>\theta_s\right\}&=\int_{x-T}^{x}dH_s(T,x,z)=\\
&=\mathrm{P}\left\{\xi ^-(\theta_s)>x-T \right\}
-q_+(s)\int_{x-T}^{0}e^{\rho _+(s)(y-(x-T))}dP_-(s,z)+
\end{aligned}\\
+Q^T(s,x)\Biggl[  \int_{-T}^{0}e^{\rho _+(s)(z+T)}dP_-(s,z)
-\mathrm{P}\left\{\xi ^-(\theta_s)>-T\right\}\Biggr ],
\end{multline*}
 and using formula~\rm{(9)}, we
obtain~\rm{(23)} after some simple transformations .
Substituting~\rm{(23)} into the first relation of~\rm{(24)} we find
the m.g.f. of $\tau (x,T)$, and then we can get the m.g.f. of $\tau
^-(x,T)$ (see the second relation in~\rm{(24)}).
\end{proof}
\par On the basis of formulas~\rm{(6)}~-~\rm{(8)} we can get the next statement about the
limit behavior of $Q^T(s,x)$ and $h_s(T,x,z)$, as $s\rightarrow 0$.
\begin{slid} Function $h'_0(T,x,z)=\lim_{s\rightarrow
0}s^{-1}h_s(T,x,z)$ $(x-T<z<x,\;z\neq 0,\; 0<x<T)$ according to the
sign of $m$ have the next forms:
\par \noindent if $m>0$
\begin{equation} h'_0(T,x,z)=\frac{1}{m}\left(
c^{-1}\frac{\partial}{\partial z}\mathrm{P}\left\{\xi
^-<z\right\}-\mathrm{P}\left\{\xi ^-
>z\right\}\right)
I\left\{z<0\right\}+\\
+\frac{1}{m}Q^T(x)\mathrm{P}\left\{\xi ^->z-x\right\};
\end{equation}
%
if $m<0$
\begin{multline}
h'_0(T,x,z)=\left(-p_+\frac{\partial}{\partial z}\mathrm{E}\,\tau
^-(z) +q_+\rho _+\int_{z}^{0}e^{\rho _+(y-z)}d\mathrm{E}\,\tau ^-(y)
\right)I\left\{z<0\right\}-\\
-Q^T(x)\rho _+\int_{z-x}^{0}e^{\rho _+(y-(z-x))}d\mathrm{E}\,\tau
^-(y);
\end{multline}
%
if $m=0$
\begin{multline}
h'_0(T,x,z)=\left(-\frac{\partial}{\partial z}\mathrm{E}\,\tau
_0(z)-c\lambda ^{-1}+c\int_{z}^{0}\frac{\partial}{\partial
y}\mathrm{E}\,\tau _0(y)dy \right)I\left\{z<0\right\}+\\
+cQ^T(x)\left(\lambda ^{-1}-\int_{z-x}^{0}\frac{\partial}{\partial
y}\mathrm{E}\,\tau _0(y)dy \right).
\end{multline}
%
The ruin probability
$$Q^T(x)=\lim_{s\rightarrow 0}Q^T(s,x)$$
 $($according to the sign of $m$$)$ is determined from~{\rm{(9)}} in the following way
\begin{equation}  Q^T(x)=
  \begin{cases}
   \displaystyle \int_{x-T}^{0}d\mathrm{P}\left\{\xi ^-<y\right\}\times\\
   \displaystyle\quad\quad  \times\biggl[  \int_{-\infty}^{-T}e^{c(T+y)}
                                       d\mathrm{P}\left\{\xi ^-<y\right\}+
   \int_{-T}^{0}d\mathrm{P}\left\{\xi ^-<y\right\}\biggr]^{-1},\; m>0,\\
   \phantom{.}\\
 \displaystyle q_+e^{-\rho _+x}\left(\frac{1}{\lambda p_+}-\int_{x-T}^{0}e^{\rho _+y}
 \frac{\partial}{\partial y}\mathrm{E}\,\tau ^-(y)dy\right)\times\\
 \displaystyle  \times\biggl[ \frac{1}{\lambda p_+}- e^{-\rho _+T}
 \int_{-\infty}^{-T}e^{c(T+y)}\frac{\partial}{\partial y}\mathrm{E}\,\tau ^-(y)dy-\\
  \displaystyle\quad \quad \quad \quad \quad \quad \quad -
  \displaystyle\int_{-T}^{0}e^{\rho _+y}\frac{\partial}{\partial y}
  \mathrm{E}\,\tau ^-(y)dy\biggr]^{-1},\; m<0,\\
  \phantom{.}\\
  \displaystyle \left(\lambda ^{-1}-\int_{x-T}^{0}\frac{\partial}{\partial
y}\mathrm{E}\,\tau _0(y)dy \right)\times\\
  \displaystyle\quad \times\biggl [\lambda ^{-1}-
  \int_{-\infty}^{-T}e^{c(T+y)}\frac{\partial}{\partial
y}\mathrm{E}\,\tau _0(y)dy - \int_{-T}^{0}\frac{\partial}{\partial
y}\mathrm{E}\,\tau _0(y)dy\biggr]^{-1},\; m=0.
  \end{cases}
\end{equation}
The distribution of $\xi (\tau ^-(x,T))$  has the next form:
\begin{equation}
  \begin{aligned}
  \mathrm{P}\left\{\xi (\tau ^-(x,T))< z,\;A_-(x)\right\}=&\frac{1}{\lambda }\Pi _-(z)
  +\int_{x-T}^{0-}\Pi_-(z-y)h'_0(T,x,y)dy+\\&+\int_{0+}^{x}\Pi_-(z-y)h'_0(T,x,y)dy,\;z<x-T.
\end{aligned}
\end{equation}
\end{slid}
\begin{slid} For the process $\xi (t)$ with the cumulant
function
\begin{equation}
   \psi (\alpha )=\lambda p(c(c-i\alpha )^{-1}-1)+\lambda q(b(b+i\alpha )^{-1}-1),
\end{equation}
 $Q^T(x)$ is represented in the following way $(0<x<T)$
\begin{equation}
  Q^T(x)=\begin{cases}
          \displaystyle\left(1-q_-e^{\rho _-(x-T)} \right)
          \left(1-q_-c\left(c+\rho _- \right)^{-1}e^{-\rho _-T}
          \right)^{-1},\;m>0,\\
          \phantom{.}\\
          \displaystyle q_+e^{-\rho _+x}\left(1-b(\rho _+ +b )^{-1}
          e^{\rho _+(x-T)} \right)\left(1-b(\rho _+ +b )^{-1}q_+e^{-\rho _+T}
          \right)^{-1},\;m<0,\\
          \phantom{.}\\
          \displaystyle \frac{c(1+b (T-x))}{b +c + bc T},\;m=0.
         \end{cases}
\end{equation}
If $\xi (t)$ is a symmetric process $(p=q=1/2,\,b=c)$, then
$$Q^T(x)=\frac{1+c(T-x)}{2+cT},\;Q_T(x)=\frac{1+cx}{2+cT},\quad (0<x<T).$$
\end{slid}
\begin{proof} Let's note that the process with cumulant~\rm{(30)} is
the almost upper and lower semi-continuous process. Then in addition
to relations~\rm{(4)}~-~\rm{(5)} we have that
\begin{equation}
   \varphi_-(s,\alpha
)=\dfrac{p_-(s)(b+i\alpha)}{\rho_-(s)+i\alpha},
\end{equation}
where $-\rho_-(s)=-bp_-(s)$ is the negative root of the equation $\psi (-i r)=s$, $s>0$,
\begin{equation}
 \mathrm{P}\left\{\xi ^-(\theta_s)<x\right\}= T^-(s,x)=q_-(s)e^{\rho_-(s)x}, x<0.
\end{equation}
 If $m>0$, then
\begin{equation}
    \mathrm{P}\left\{\xi ^-(\theta_s)<x\right\}\underset{s\to0}{\longrightarrow}
    \mathrm{P}\left\{\xi  ^-<x\right\}=q_-e^{bp_-x},\;x<0,\quad
    p_-(s)\underset{s\to0}{\longrightarrow}p_->0.
\end{equation}
Taking into account that $p_+(s)p_-(s)=s\left(s+\lambda
\right)^{-1}$, we have, for
 $m<0,$
$q'_-(s)=-p'_-(s)\rightarrow-(\lambda p_+) ^{-1}$ as $s\rightarrow
0$. Hence,
\begin{equation}
  \mathrm{E}\,\tau ^-(x)=-\frac{\partial}{\partial s}T^-(s,x)|_{s=0}=\frac{1-bx}{\lambda
  p_+},\;x<0.
\end{equation}
If $m=0$, then for $\widetilde{\xi }_0(t)$, we have $\Pi
_0(dx)=\lambda _0 b e^{b x}dx,\, x<0,\; \lambda _0=\lambda
q(c+b)b^{-1}$, moreover
$$
\widetilde{\xi }^-_0(t)=\widetilde{\xi }_0(t),\;p^0_-(s)=
\mathrm{P}\left\{\widetilde{\xi }_0(\theta_s)=0\right\}
=\frac{s}{s+\lambda _0}.
$$
Hence, the m.g.f. of $\tau _0(x)$ has the form
$$T^-_0(s,x)=\mathrm{E}\,e^{-s\tau _0(x)}=q_-^0(s)e^{bp_-^0(s)x},\;x<0.$$
Since $(p_-^0)'(s)=-(q_-^0)'(s)\rightarrow \lambda _0^{-1}$ as
$s\rightarrow 0$, we get
\begin{equation}
\mathrm{E}\,\tau _0(x)=-\frac{\partial}{\partial
s}\left.T^-_0(s,x)\right |_{s=0}=\frac{1-bx}{\lambda _0},\;x<0.
\end{equation}
 Substituting formulas~\rm{(34)}~-~\rm{(36)}
into the corresponding relations of~\rm{(28)} we get~\rm{(31)}.
\end{proof}

\smallskip\noindent{\it{Remark}}. We should note that it is easy to
get the representation of the m.g.f. of the functionals related to
the exit from the interval for the almost lower semi-continuous
process $\eta (t)$ (with the parameter $b>0$, by considering that
$\xi(t)=-\eta(t)$). Particularly,
\begin{multline}
Q_T(s,x)=q_-(s)\int_{0}^{x}e^{\rho
_-(s)(x-y)}dP_+(s,y)\times\\
\times\left [ \int_{T}^{\infty}e^{b
(T-y)}dP_+(s,y)+\int_{0}^{T}e^{\rho _-(s)(T-y)}dP_+(s,y)\right
]^{-1}.
\end{multline}
\smallskip
\par Let $\xi (t)$ be the almost upper semi-continuous piecewise constant
process.
 Then $\xi _1(t)=at+\xi (t)$, $a<0$ is the almost upper
 semi-continuous piecewise linear process.
For the process $\xi _1(t)$ on the basis of the stochastic relations for $\tau ^+(x,T)$:
$$
  \tau ^+(x,T)\dot{=}
  \begin{cases}
    \zeta ,\; \xi +a\zeta>x, \\
    \zeta +\tau ^+(x-\xi -a\zeta ),\; x-T<\xi +a\zeta <x,
  \end{cases}
$$
we have the next integro-differential equation for $Q^T(s,x)$
\begin{multline}
  a\frac{\partial}{\partial x}Q^T(s,x)=\lambda
  \int_{x-T}^{x}Q^T(s,x-z)dF_1(z)-(s+\lambda )Q^T(s,x)+\lambda
  pe^{-cx},\; 0<x<T.
\end{multline}
Introducing the function
$\overline{Q}\phantom{|}^T(s,x)=1-Q^T(s,x)$, and following the
reasoning analogous to that for the piecewise constant process $\xi
(t)$ we can get the representation of the functionals related to the
exit from the interval $\left (  x-T,\,x\right )$ for the piecewise
linear processes.
\par The two boundary problems for the integer - valued random-walks
are considered in~\cite{10} and for the process with stationary
independent increments are treated in~\cite{11}.
\enddocument